\documentclass{amsart}
\usepackage{amsmath}
\usepackage{graphicx}
\usepackage{latexsym}

\newtheorem{thm}{Theorem}

\newtheorem{theorem}[thm]{Theorem}
\newtheorem{lemma}[thm]{Lemma}
\newtheorem{corollary}[thm]{Corollary}
\newtheorem{proposition}[thm]{Proposition}

\theoremstyle{definition}

\newtheorem{remark}[thm]{Remark}

\newtheorem{question}[thm]{Question}

\newcommand{\CPb}{\overline{\mathbb{CP}}{}^{2}}
\newcommand{\CP}{{\mathbb{CP}}{}^{2}}

\newcommand{\R}{\mathbb{R}}
\newcommand{\Z}{\mathbb{Z}}
\newcommand{\Q}{\mathbb{Q}}
\newcommand{\C}{\mathbb{C}}

\def \x {\times}
\def \eu{{\text{e}}}
\def \sign{{\text{sign}}}
\newcommand{\ste}{S^2 \times S^2}
\newcommand{\sto}{S^2 \tilde{\times} S^2}

\begin{document}

\title[Round handles, logarithmic transforms, and smooth $4$-manifolds]
{Round handles, logarithmic transforms, \\ and smooth $4$-manifolds}

\author[R. \.{I}nan\c{c} Baykur]{R. \.{I}nan\c{c} Baykur}
\address{Max Planck Institute for Mathematics, Bonn, Germany \newline
\indent Department of Mathematics, Brandeis University, Waltham MA, USA}
\email{baykur@mpim-bonn.mpg.de, baykur@brandeis.edu}

\author[Nathan Sunukjian]{Nathan Sunukjian}
\address{Stony Brook University, Department of Mathematics} 
\email{nsunukjian@math.sunysb.edu}

\begin{abstract} 
Round handles are affiliated with smooth $4$-manifolds in two major ways: $5$-dimensional round handles appear extensively as the building blocks in cobordisms between $4$-manifolds, whereas $4$-dimensional round handles are the building blocks of broken Lefschetz fibrations on them. The purpose of this article is to shed more light on these interactions. We prove that if $X$ and $X'$ are two cobordant closed smooth $4$-manifolds with the same euler characteristics, and if one of them is simply-connected, then there is a cobordism between them which is composed of round $2$-handles only, and therefore one can pass from one to the other via a sequence of generalized logarithmic transforms along tori. As a corollary, we obtain a new proof of a theorem of Iwase's, which is a $4$-dimensional analogue of the Lickorish-Wallace theorem for $3$-manifolds: Every closed simply-connected $4$-manifold can be produced by a surgery along a disjoint union of tori contained in a connected sum of copies of $\CP$, $\CPb$ and $S^1 \x S^3$. These answer some of the open problems posted by Ron Stern in \cite{S}, while suggesting more constraints on the cobordisms in consideration. We also use round handles to show that \emph{every} infinite family of mutually non-diffeomorphic closed smooth oriented simply-connected $4$-manifolds in the same homeomorphism class \emph{constructed up to date} consists of members that become diffeomorphic after one stabilization with $\ste$ if members are all non-spin, and with $\ste \, \# \, \CPb$ if they are spin. In particular, we show that simple cobordisms exist between knot surgered manifolds. We then show that generalized logarithmic transforms can be seen as standard logarithmic transforms along fiber components of broken Lefschetz fibrations, and show how changing the smooth structures on a fixed homeomorphism class of a closed smooth $4$-manifold can be realized as relevant modifications of a broken Lefschetz fibration on it. 
\end{abstract}
\maketitle

\section{Introduction}

During the past three decades, great advances have been made in the study of smooth $4$-manifolds, demonstrating highly peculiar phenomena in vast families of examples, examples which devastated many proposed classification schemes. Most of these examples involve the ``logarithmic transform'' operation, which is the $4$-dimensional analogue of the Dehn surgery operation in dimension $3$. In particular, all known constructions of \textit{infinite} families of smooth structures on a fixed homeomorphism type involve logarithmic transforms along tori. (See \cite{FS6} for an excellent survey on this subject.) Hence, Ron Stern posted the following open problems in \cite{S}: \ (P1) Are any two arbitrary closed smooth oriented simply-connected $4$-manifolds $X$ and $X'$ in the same homeomorphism class related via a sequence of logarithmic transforms along tori? (Problem 12 in \cite{S}.) \linebreak (P2) Is there a cobordism between $X$ and $X'$ which is composed of round $2$-handles only? (Problem 15 in \cite{S}.) 

In Section \ref{Round cobordisms}, we answer these problems affirmatively. We begin in Section \ref{PreliminaryResults} by observing the effect of attaching $5$-dimensional round handles to $4$-manifolds. We then prove that between any two cobordant closed smooth $4$-manifolds, one of which is simply-connected, there exists a cobordism with only round $2$-handles (Theorem \ref{RoundCobordismThm}). Our result builds up on Asimov's construction of round handle decompositions in \cite{A}. It will follow that one can pass from one of these $4$-manifolds to the other via a sequence of generalized logarithmic transforms, and also that any closed simply-connected $4$-manifold $X$ can be produced by a sequence of logarithmic transforms on a connected sum of $\CP$ and $\CPb$, where the numbers of summands are determined by the euler characteristic and signature of $X$ (Corollaries \ref{related} and \ref{relatedbylog}). In turn, we obtain the following corollary: Any simply-connected $4$-manifold $X$ can be obtained by performing \textit{simultaneous} generalized logarithmic transforms along a framed `link' of embedded self-intersection zero tori in connected sums of $\CP$, $\CPb$ and $S^1 \x S^3$ (Corollary \ref{simultaneous}). After posting the first version of our article, we found out that this last result was proved earlier by Iwase in \cite{Iwase}, who did not use the language of logarithmic transforms at the time, from which one can derive Corollary \ref{relatedbylog} as well. Unlike in our proof, Iwase's arguments make use of the famous result of C.T.C. Wall's mentioned below. These results are analogous to the Lickorish-Wallace theorem for $3$-manifolds, which states that any closed orientable $3$-manifold can be obtained by Dehn surgeries on a framed link in the $3$-sphere. Here we shall note that in the cobordisms we obtain, stabilizations/destabilizations with standard pieces $\ste$ and $\sto$ appear very often in the intermediate steps, obstructing one's hope of relating the well-known smooth invariants of the $4$-manifolds on the two ends of the cobordism using standard gluing arguments. We therefore include a discussion about further constraints that might be imposed on these cobordisms to avoid this (see Remark \ref{caution}) and we discuss different versions of our results under such extra assumptions. 

A celebrated theorem of C.T.C. Wall says that any two closed oriented smooth simply-connected $4$-manifolds $X$ and $X'$ homeomorphic to each other become diffeomorphic after stabilizing with some number of copies of $\ste$ \cite{Wa}, raising the question whether $X$ and $X'$ become diffeomorphic after only one stabilization or not. (Problem 14 in \cite{S}.) Let $ \{X_m \, | \, m: 1, 2, \ldots \} $ be a family of mutually non-diffeomorphic closed smooth oriented simply-connected 4-manifolds in the same homeomorphism class. If $X_{m+1}$ is obtained from $X_m$ by a log transform, then we show that they become diffeomorphic after one stabilization with either $\ste$ or $\sto$. Consequently, \emph{every} infinite family of mutually non-diffeomorphic closed smooth oriented simply-connected \emph{non-spin} $4$-manifolds in the same homeomorphism class \emph{constructed up to date} consists of members that become diffeomorphic after one stabilization with $\ste$, and similarly after stabilizing with $\sto$ (since $X \# \, \sto$ is diffeomorphic to $X \# \, \ste$ when $X$ is non-spin). If instead we have a family of spin $4$-manifolds, the same holds true after blowing-up all once. In particular, we obtain a new and simple proof of knot surgered $4$-manifolds becoming diffeomorphic after just one stabilization (Corollary \ref{logrelevance}), a result originally due to Auckly \cite{Au} and independently to Akbulut \cite{Ak}. We also show that the unknotting number of a knot $K$ does not give a lower bound on the number of logarithmmic transforms one needs to pass from a $4$-manifold $X$ to a knot sugered manifold $X_K$.

Lastly, in Section \ref{BLFs} we turn to broken Lefschetz fibrations over the $2$-sphere, whose building blocks are $4$-dimensional round handles. We show that \textit{generalized} logarithmic transforms can be seen as \textit{standard} logarithmic transforms along fiber components of broken Lefschetz fibrations (i.e. the logarithmic transform becomes `standard' in a `generalized' fibration), and present how changing the smooth structure on a fixed homeomorphism class of a closed smooth $4$-manifold can be realized as a modification of a broken Lefschetz fibration on it (Theorem \ref{BLFcobordismThm}).

\vspace{0.2in}
\section{Preliminaries} \label{Preliminaries}

\subsection{Round handles} \ 

An \emph{$m$-dimensional round $k$-handle} is simply $S^1$ times an $(m-1)$-dimensional $k$-handle, i.e. an $S^1 \x D^k \x D^{m-k-1}$ attached along $S^1 \x \partial D^k \x D^{m-k-1}$. An $m$-dimensional round $k$-handle $R_k$ decomposes as the attachment of two $m$-dimensional handles $h_{k-1}$ and $h_k$ of indices $k-1$ and $k$, which intersect algebraically zero times but geometrically twice. This can be easily visualized by looking at the decomposition of an annulus into two handles, and one then thickens the cores and cocores to the desired dimensions. If $R_k$ can be decomposed into $h_{k-1}$ and $h_k$ where one can isotope these handles so that they are attached independently, i.e. the core of $h_k$ does not intersect the cocore of $h_{k-1}$, then we will call $R_k$ a \textit{trivial round handle}. 

We will use the following notational conventions throughout this article: For $Y$ and $Y'$ compact oriented manifolds with boundary, $Y=Y'$ means that there is an orientation preserving diffeomorphism between $Y$ and $Y'$. If $W$ is an oriented cobordism from $X$ to $X'$ with a given handlebody decomposition, 
\[ W = I \times X + \Sigma \, {h_1} + \ldots + \Sigma \, {h_n} \, ,\]
such that $X = \partial_{-} W$ and $X' = \partial_{+} W$, by abuse of notation we will write 
\[X' = X + \Sigma \,  {h_1} + \ldots + \Sigma \, {h_n} \, .\]
Similarly, when $W$ has a round handle decomposition 
\[W = I \x X + \Sigma \, {R_1} + \ldots + \Sigma \, {R_n} \, , \]
we will write 
\[ X' = X + \Sigma \, {R_1} + \ldots + \Sigma \, {R_n} \, .\]  
Here $\Sigma h_i$ and $\Sigma R_i$ denote a collection of index $i$ handles and a collection of index $i$ \textit{round} handles, respectively.

As observed by Asimov \cite{A}, any handle pair $h_{k-1}$ and $h_k$ attached to a manifold $X$ independently can be turned into a trivial round $k$-handle, i.e. 
\[X' = X + h_{k-1} + h_k = X + R_k \]
under these assumptions. This fact is referred as the ``Fundamental Lemma of Round Handles'' in \cite{A}. Lastly, note that if an $m$-manifold $X$ decomposes into round handles, then it necessarily has trivial euler characteristic. As shown by Asimov, this is not only a necessary but also a  sufficient condition provided that $m \neq 3$ and $X$ is not the M\"obius band \cite{A}.

\vspace{0.2in}
\subsection{Logarithmic transforms} \

Let $T$ be an embedded $2$-torus in a $4$-manifold $X$ with trivial normal bundle $\nu T$. A \emph{framing} of $\nu T$ is the choice of a projection  $\pi: \nu T \to D^2$, which equivalently is a choice of an orientation-preserving diffeomorphism $\tau :\nu T \to D^2 \x T^2$, resulting in an identification \[ H_1(\partial (X \setminus \nu T)) \cong H_1(T)\oplus \Z \, , \]
where the last summand is generated by a positively oriented meridian $\mu_T$ of $T$. We can construct a new $4$-manifold $X' = (X \setminus \nu T) \, \cup_{\phi} D^2 \x T^2$ using a diffeomorphism $\phi: \partial (T^2 \x D^2) \to \partial \nu T$. This diffeomorphism is uniquely determined up to isotopy by the homology class 
\[ [(\tau \phi(\partial D^2)] = p [\mu] + q [\alpha] \, , \]
where $\alpha$ is a \emph{push-off} of a primitive curve in $T$ by the chosen framing $\tau$, which we will also denote by $\alpha$. To sum up, the result of the surgery is determined by the torus $T$, the framing $\tau$ (equivalently $\pi$), \emph{surgery curve} $\alpha$ and \emph{the surgery coefficient} $p/q \in \Q \cup \{ \infty \}$. We will encode this data in the notation $X(T,\tau, \alpha, p/q)$, and call this operation producing $X'= X(T, \tau, \alpha, p/q)$ as the \emph{generalized logarithmic $p/q$ transform} of $X$ along $T$ with framing $\pi$. If $q=\pm 1$, we will refer to it as an \textit{integral} logarithmic transform, otherwise we call it a \textit{rational} logarithmic transform. It shall be clear from the very definitions that a logarithmic $\infty$ transform gives $X$ back, and a logarithmic $0$ transform usually changes the topology. 

The above operation generalizes the \emph{standard logarithmic transform} performed on an elliptic surface, or on the total space of a genus one Lefschetz fibration, which amounts to modifying the $4$-manifold along with the fibration on it by replacing a regular torus fiber with an $m$-multiple of this fiber. The new fibration conforms to the local model: 
\[ (D^2 \x T^2) / \Z_m \to D^2 / \Z_m \, , \]
where the generator $\sigma \, $ of $\Z_m$ acts on $D^2 \x T^2$ by 
\begin{equation} \label{standardlog}
\sigma \, (z, x, y) = (exp(2 \pi i/m) \, z, \, x - p/m, \, y - q/m) 
\end{equation}
for $(z,x,y) \in \C \x \R^2/{\Z^2}$ with $|z|=1$, $gcd(m,p,q)=1$, and acts on $D^2$ by 
\[z \mapsto exp(2 \pi i/m)z \, , \]
inducing a fibration coming from the projection of $D^2 \x T^2$ onto its $D^2$ component. That is, the standard logarithmic transform is defined for such an $(X,f)$, with $T=f^{-1}(z)$ for some regular value $z$ of $f$, where $\pi=f$ and the surgery coefficient is always integral (and $p,q$ in the above local model are auxiliary). Below, we reserve the expression \textit{log transform} for generalized logarithmic transform and say \textit{standard} log transform to indicate this special setting.

\vspace{0.2in}
\subsection{Broken Lefschetz fibrations} \

Let $X$ and $\Sigma \, $ be closed oriented manifolds of dimension four and two, respectively, and $f : X \to \Sigma \, $ be a smooth surjective map. The map $f$ is said to have a Lefschetz singularity at a point $x$ contained in a discrete set $C \subset X$, if around $x$ and $f(x)$ one can choose orientation preserving charts so that $f$ conforms the complex local model $(u, v) \to u^2 + v^2$. The map $f$ is said to have a round singularity along an embedded $1$-manifold $Z \subset X \setminus C$ if around every $z \in Z$, there are coordinates $(t, x_1, x_2, x_3)$ with $t$ a local coordinate on $Z$, in terms of which $f$ is given by $(t, x_1, x_2, x_3) \to (t, x_1^2 + x_2^2 - x_3^2)$. A broken Lefschetz fibration is then defined as a smooth surjective map $f : X \to \Sigma \, $ which is a submersion everywhere except for a finite set of points C and a finite collection of circles $Z \subset X \setminus C$, where it has Lefschetz singularities and round singularities, respectively. These fibrations are found in abundance, as any generic map from $X$ to $S^2$ can be homotoped to a broken Lefschetz fibration \cite{Sae, B}.

\vspace{0.2in}
\section{Round cobordisms and logarithmic transforms} \label{Round cobordisms}

\subsection{Preliminary results.}  \label{PreliminaryResults} \
Many of the new smooth $4$-manifolds that have arisen in the past couple of decades are constructed using similar techniques: They use logarithmic transforms and fiber sums to produce new smooth structures on a fixed homeomorphism type of a $4$-manifold. Before we discuss the role of round handles in cobordisms between homeomorphic simply-connected $4$-manifolds, let us first demonstrate how round handles appear in fiber sums.

\begin{proposition} \label{fibersum}
Let $\Sigma \, _i$ be closed orientable surfaces of genus $g$ with trivial normal bundle in $X_i$, $i=1,2$, and $X$ be a fiber sum of $X_1$ and $X_2$ along $\Sigma \, _1$ and $\Sigma \, _2$. Then $X$ is obtained from the disjoint sum of $X_1 \setminus \nu T_1$ and $X_2 \setminus \nu T_2$ by attaching round handles. 
\end{proposition}

\begin{proof} We get $X$ from the disjoint union $(X_1 \setminus \nu T_1) \, \sqcup (X_2 \setminus \nu T_2)$ by attaching an \linebreak $S^1 \x I \x \Sigma \, _g$ to the latter. Here, the basic handle decomposition of 
\[ \Sigma \, _g= h_0 + \Sigma \, _{i=1}^{2g} h_1^i + h_2 \]
yields a decomposition of $S^1 \x I \x \Sigma \, _g$ into round handles $R_0$, $R_1^1,\ldots R_1^{2g}$, $R_2$. 
\end{proof}

On the other hand, log transforms, which will be of our main focus in this article, are related to round handles as follows:

\begin{lemma} \label{round2islog}
A round $2$-handle attachment to a $4$-manifold $X$ is equivalent to performing an integral generalized logarithmic transformation on $X$.
\end{lemma}

\begin{proof}
When one attaches a round $2$-handle to $X$, the effect is to surger out $S^1 \x S^1 \x D^2$ and glue back in $S^1 \x D^2 \x S^1$ to obtain a new $4$-manifold $X'$. Thus, the attachment of a round $2$-handle to $X$ is nothing but a log transform along a torus $T \subset X$ identified with the \textit{attaching torus} of the round $2$-handle. We will show that any integral logarithmic transform can be realized by such a round handle attachment.

Let $T$ be an embedded $2$-torus in a $4$-manifold $X$ with trivial normal bundle $\nu T$ framed by $\pi: \nu T \to D^2$, and let $\alpha \subset T$ be the surgery curve. The new manifold $X'=X(T,\tau, \alpha, p)$ is obtained by attaching one $2$-handle, two $3$-handles and a $4$-handle to $\partial(X \setminus \nu T)$, and therefore determined by the framed attaching circle of the $2$-handle. Let $\varphi \, $ be a self-diffeomorphism of $T^2 =S^1 \x S^1$ such that $\{pt\} \x S^1$ is mapped to $\alpha$ and $S^1 \x \{pt\}$ is mapped to some primitive curve $\beta$ on $T$, and set $\tau'= (\varphi \, ^{-1} \x id) \circ \tau$. Hence, $X'=X(T,\tau, \alpha, p)=X(T,\tau', \alpha', p)$, where $\alpha'$ is the new surgery curve. Using this new framing, we can glue a round $2$-handle $R=S^1 \x D^2 \x D^2$ to $X$ along $\nu T$ such that $S^1 \x \{ pt \} \x \{0 \}$ is mapped to $\beta$ and $\{ pt \} \x S^1 \x \{ 0 \}$ is mapped to $\alpha'$. The attachment of $R$ is given by $S^1$ times a $2$-handle attachment, and for each $x \in S^1$ we attach $\{x\} \x D^2 \x D^2$ so that it maps to \, $[\{x'\}] \x [\alpha'] \x p[\mu]$ \, in the homology, where $x' \in \beta$. In other words, the gluing of $R$ has the effect of a family of integral Dehn surgeries parametrized by $x' \in \beta$. We conclude that $X'=X(T,\tau, \alpha, p) = X + R$.  
\end{proof}

\begin{remark}
A slight generalization of round handles are ``twisted round handles'', where one attaches $S^1 \x D^{m-1}$ along a twisted $D^k$ bundle over $S^1$. These appear naturally in broken Lefschetz fibrations on $4$-manifolds. As shown in \cite{B2}, the twisted round handles and the regular ones we focus on in this article are the only ones to which we can attach a family of $k$-handles parametrized along $S^1$. It is easy to see that $5$-dimensional twisted round $2$-handles are equivalent to \textit{Klein bottle surgeries} (see for example \cite{N}) in the same fashion as above.
\end{remark}

We will make repeated use of the following: Say $W$ is a cobordism from $X$ to $X'$ with a given handlebody decomposition, 
\[W = I \times X + \Sigma \, {h_1} + \ldots + \Sigma \, {h_n}\]
or a round handle decomposition 
\[W = \Sigma \, {R_1} + \ldots + \Sigma \, {R_n} \, , \]
both given by increasing indices. By looking at the `dual' decomposition of the handles of $W$ of index greater than i: 
\[W = (I \times X + \Sigma \, {h_1} + \ldots + \Sigma \, {h_i}) \cup (\Sigma \, {h_{i+1}^*} + \Sigma \, {h_n^*} + I \times X') \, , \]
we see that 
\[Y= \partial_+(I \times X + \Sigma \, {h_1} + \ldots + \Sigma \, {h_i}) \, \]
and 
\[Y'= \partial_+(I\times X' + \Sigma \, {h_{n}^*} + \ldots + \Sigma \, {h_{i+1}^*}) \]
are diffeomorphic. In this case, we will say ``$Y$ and $Y'$ can be seen to be diffeomorphic by looking at the i-th level of $W$''.

The next few lemmas will follow from the correspondence given in Lemma \ref{round2islog}:  

\begin{lemma} \label{stabilizelemma}
If $X$ and $X'$ are simply-connected $4$-manifolds, and $X'$ is the result of performing an integral log transform on $X$, then $X$ and $X'$ become diffeomorphic either after stabilizing each with $\ste$ or with $\sto$. If in addition $X$ and $X'$ are both non-spin, then both $X \# \, S^2 \times S^2 = X' \# \, S^2 \times S^2$ and $X \# \, \sto = X' \# \, \sto$.
\end{lemma} 
 
\begin{proof}

By Lemma \ref{round2islog}, manifolds related by an integral log transform are cobordant by a cobordism $W$ consisting of a single round $2$-handle, which in turn decomposes as a pair of $2$- and $3$-handles. The lemma will follow from looking at the middle level of $W$.

By attaching a $5$-dimensional $2$-handle to a $4$-manifold $X$, we surger out an $S^1 \x D^3$ and glue in a $D^2 \x S^2$. When $X$ is simply-connected, this amounts to connect summing $X$ with $\ste$ or $\sto$. By looking at the middle level of $W$, we see that $X$ and $X'$ become diffeomorphic after connect summing each with $\ste$ or $\sto$. Note that if $X$ and $X'$ are both non-spin, then $X \# \ste = X \# \sto$ and $X' \# \ste = X' \# \sto$, so one can either connect sum both with $\ste$ or both with $\sto$ to obtain the diffeomorphism. If $X$ and $X'$ are both spin, then they can become homeomorphic only after connect summing both with $\ste$ or both with $\sto$ but not in a mixed way.
\end{proof}

\begin{theorem} \label{towertheorem}
Let $\bigsqcup{_{i=0}^M} T_i$ be a collection of pairwise disjoint embedded tori with self-intersection zero in a simply-connected $4$-manifold $X$. Let $X_m$ denote the result of successively performing log transforms on the tori $T_1, \ldots, T_m$, with $X' = X_M$, and assume that all $X_m$ simply-connected. If in addition every $X_m$ is non-spin for $m=1, \ldots M$ , then $X \# \, S^2 \times S^2 = X' \# \, S^2 \times S^2$ and $X \# \, \sto = X' \# \, \sto$. If $X_m$ are not all spin, then, $X \# \, \ste \# \CPb = X' \# \, \ste \# \CPb$.
\end{theorem}

\begin{proof}
Say each $X_m$ is non-spin. Then $X_m \# \, \ste = X_{m+1} \# \, \ste$  by Lemma \ref{stabilizelemma}, for all $m= 1, \ldots M-1$. By induction, $X \# \, S^2 \times S^2 = X' \# \, S^2 \times S^2$. If they are not all spin, then after blowing-up each $X_m$, we get a family of homeomorphic non-spin $4$-manifolds and apply the same argument, noting that $\ste \, \# \CPb = \sto \, \# \CPb$. 
\end{proof}

\begin{remark}
It is possible to generalize this theorem to the case when \textit{rational} logarithmic transforms are involved. To see this, one first observes that a rational log transform can be expressed as a sequence of integral log transforms, obtained using continued fractions analoguous to the $3$-dimensional case, and then replaces the given cobordism by one composed of integral log transforms. However, we will not need this generalization for the results that follow.  
\end{remark}

Lastly, we point out a special situation where one can trade a round $1$-handle with a round $2$-handle. (Also see Proposition \ref{TrivialLog} below.)

\begin{lemma} \label{KillRound1}
Let $W$ be a cobordism between $4$-manifolds $X$ and $X'$ given by a round $1$-handle which is attached along two oriented loops which are homotopic with the same orientation. Then there is a cobordism $W'$ between $X$ and $X'$ which is given by a round $2$-handle.
\end{lemma}

\begin{proof}
Since $R_1$ is attached along two oriented loops that are oriented homotopic, we can consider them as being attached along $S^1 \times \left\{pt_1, pt_2\right\}$ contained in some $S^1 \times D^3 \subset X$. (We assume the loops $S^1 \x \{0\} \x D^3$ and $S^1 \x \{1 \} \x D^3$ to be co-oriented, following a choice of orientation on the first $S^1$ component.) If we attach a $4$-dimensional $1$-handle to $D^3$ along two points $pt_1, pt_2 \in D^3$, the result is $(S^1 \x S^2) \backslash D^3$. Since $R_1$ is just a $4$-dimensional $1$-handle times $S^1$, the result of attaching a round $1$-handle to $S^1 \times D^3$ along $\left\{pt_1, pt_2\right\} \times S^1$ is an $S^1 \times S^2 \setminus D^3$ bundle over $S^1$. There are two such bundles: the trivial bundle, and the mapping torus constructed using the self-diffeomorphism $\phi:S^1 \times S^2 \longrightarrow S^1 \times S^2$ which is defined to be the identity on $S^1$ and the antipodal map on $S^2$. We will refer to these manifolds as $S^1 \times (S^1 \times S^3 \setminus D^3)$ and $S^1 \tilde{\times} (S^1 \times S^3 \setminus D^3) = [0,1] \times (S^1 \times S^2 \setminus D^3) \, /  (0,z) \sim (1, \phi(z))$, respectively.

We will now see how both of these manifolds can also result from attaching a round $2$-handle to $S^1$ cross the unknot in $S^1 \times D^3$. Notice that if we attach a $4$-dimensional $0$-framed $2$-handle to $D^3$ along an unknot, the result is $(S^1 \times S^2) \backslash D^3$. Since a $5$-dimensional round $2$-handle is a $4$-dimensional $2$-handle times $S^1$, we can attach a round $2$-handle $R_2$ to $S^1 \times D^3$, such that the result is $S^1 \times (S^1 \times S^2 \backslash D^3)$, the same as the result of attaching a round $1$-handle.

To get $S^1 \tilde{\times} (S^1 \times S^3 \setminus D^3)$ after attaching a round $2$-handle is slightly more involved. It is perhaps easier to see the opposite direction: We will see that we can attach a round $2$-handle to  $S^1 \tilde{\times} (S^1 \times S^3 \setminus D^3)$ in such a way that the result is $S^1 \times D^3$. (Note that a round $2$-handle upside down is also a round $2$-handle.) In fact, by Lemma \ref{round2islog}, it is sufficient to find an integral logarithmic transform that accomplishes this. Let $\gamma : [0,1] \longrightarrow S^2$ be an embedding such that $\gamma (0) = \{ NP \}$ and $\gamma(1) = \{ SP \}$, where $\{\text{NP}\}$ and $\{\text{SP}\}$ stand for the north pole and the south pole of $S^2$, respectively. Then in $S^1 \tilde{\times} (S^1 \times S^3 \setminus D^3)$ we achieve our desired result by performing a logarithmic transform on the torus which is the image in the mapping torus of points of the form $( x,y,\gamma(x)) \in [0,1] \times (S^1 \times S^2 \setminus D^3)$.
\end{proof} 

\vspace{0.1in} 
\subsection{Main results on cobordisms.} \label{MainResults} \

We begin by proving a negative result:

\begin{proposition}
An h-cobordism between two closed smooth $4$-manifolds does not admit a round $2$-handle decomposition.
\end{proposition}

\begin{proof}
Suppose $W$ is a cobordism with a round $2$-handle decompositon
\[ W = X + \sum_{i=1}^m R_2^i  \]
We claim that $W$ cannot be an h-cobordism. By decomposing each round $2$-handle into a $2$- and $3$-handle, $W$ is given a regular $2$- and $3$-handle decomposition,
\[ W = X + \sum_{i=1}^m (h_2^i + h_3^i)  = X + h_2^1 + h_3^1 + \sum_{i=2}^m (h_2^i + h_3^i)\]
Recall that for an h-cobordism, the $3$-handles must algebraically cancel with the $2$-handles. In this case however, $h_3^1$ does not cancel with $h_2^1$ since together they form a round $2$-handle, which means that the core of $h_3^1$ intersects the co-core of $h_2^1$ algebraically zero times, but geometrically twice. Also $h_3^1$ does not cancel with any of the other $2$-handles since they are attached independently; the other $2$-handles are attached after $h_3^1$ and can be slid off of it. Since $h_3^1$ does not algebraically cancel with the 2-handles, $W$ cannot be an h-cobordism.
\end{proof}

It follows that the cobordism in question from problem (P2) from the introduction can never be an h-cobordism. However we do obtain:

\begin{theorem} \label{RoundCobordismThm}
Let $X$ and $X'$ be two cobordant closed smooth (oriented) $4$-manifolds with the same euler characteristic. If $X$ is simply-connected, then there exists a compact smooth (oriented) cobordism between them with round $2$-handles only.
\end{theorem}

\begin{proof}
It is a standard argument that one can eliminate all the handles with index unequal to $2$ or $3$ in any given cobordism $W$ between $X$ and $X'$. Namely, we can surger out the $1$- and $4$-handles, replacing them with $3$- and $2$-handles respectivly. Take such a simplified handle decomposition of $W$. The assumption on the euler characteristics implies that the number of $2$- and $3$-handles in this cobordism $W$ are the same. Let $N$ be this number, and note that $X = \partial_- W$, $X' = \partial_+ W$.

For each $2$-handle, check if there is a $3$-handle in $W$ that goes over it algebraically zero times and geometrically twice. Label the handles so that $h_2^i$ and $h_3^i$ for \linebreak $i= M+1, \ldots N$ are such pairs with $\partial_- h_3^i$ disjoint from $\cup_{k=i+1}^{N} \partial_+ h_2^k$. For simplicity, let us morever assume that $\partial_- h_3^i$ is disjoint from $\cup_{k=M+1}^{N} \partial_+ h_2^k$ for $i=1, \ldots M-1$ as well, so that we can write 
\[ W= \Sigma \, _{i=M+1}^N h_2^i + \Sigma \, _{i=M+1}^N h_3^i + \Sigma \, _{i=1}^{M} \, (h_2^i + h_3^i) \, . \]
Letting $\widetilde{R_2^i} = h_2^{M+i} + h_3^{M+i}$ be the round $2$-handle for $i=1, \ldots, N-M$, we set: 
\[ W_1 = I \x X + \Sigma \, _{i=1}^M h_2^i \ , \ W_2 = I \x \partial_+ W_1 + \Sigma \, _{i=1}^M h_3^i \ , \] 
\[ \text{and} \ W_3 = I \x \partial_+ W_2 + \Sigma \, _{i=1}^{N-M} \widetilde{R_2^i} \, . \] 

Next, we closely examine Asimov's proof of his main theorem in \cite{A}, which we will give a sketch of here: We can assume that the attaching spheres of the \linebreak $3$-handles are tranvserse to the belt spheres of the $2$-handles and that the attaching sphere $S_i$ of each $h_3^i$ intersects $\partial_+ X \setminus \cup_i^M \partial_- h_2^i$ at some point $p_i$. Then introduce cancelling handle pairs $H_1^i$ and $H_2^i$ away from $S_i^2$ yet in a small neighborhood of each $p_i$ contained in $\partial_+ X \cup_i^M \partial_- h_2^i$. With all these assumptions in hand, we can now use the Fundamental Lemma of Round Handles to pair up $H_1^i$ with $h_2^i$ to create round $1$-handles $R_1^i$, and $H_2^i$ with $h_3^i$ to create round $2$-handles $R_2^i$, where $i=1, \ldots, M$. We get a new handle decomposition of $W_1 \cup W_2 = W'_1 \cup W'_2$ where 
\[ W'_1 = I \x X + \Sigma \,  R_1^i \  \ \text{and} \ \ W'_2= I \x \partial_+ W'_1 + \Sigma \,  R_2^i \, . \]
An important point here is that all $R_1^i$ for $i=1, \ldots, M$ are attached independently from each other and are therefore can be thought as being attached to $I \x X$ simultaneously. (Same holds true for the attachment of $R_2^i$ to $I \x \partial_+ W'_1$ for $i=1, \ldots, M$.) Since $X$ is simply-connected, any two (oriented) loops are (oriented) homotopic in it. Therefore, we can apply Lemma \ref{KillRound1} to replace each $R_1^i$ with a round $2$-handle $\bar{R_2^i}$ for $i=1, \ldots, M$ to obtain a new cobordism $\bar{W}_1$ from $X$ to $\partial_+ W'_1$ which is composed of round $2$-handles. Hence, we get a new cobordism 
\[ \tilde{W} = \bar{W}_1 + W'_2 + W_3 = I \x X + \Sigma \,  \bar{R_2^i} + \Sigma \,  R_2^i + \Sigma \,  \widetilde{R_2^i} \, , \]
composed solely of round $2$-handles.
\end{proof}


Hence we get: 

\begin{corollary} \label{related}
Let $X$ and $X'$ be two cobordant closed smooth (oriented) $4$-manifolds with the same euler characteristic. If $X$ is simply-connected, then $X'$ can be obtained from $X$ by a sequence of log transforms along tori. Any simply-connected $4$-manifold $X$ can be obtained from $a \CP \# b \CPb$ by a sequence of log transforms along tori, where $e(X)=a+b+2$, $\sign(X)=a-b$.
\end{corollary}

\begin{proof}
The first part is immediate after Lemma \ref{round2islog} and Theorem \ref{RoundCobordismThm}. For the second part, simply observe that the signatures of $X$ and $a \CP \# b \CPb$ are the same, thus they are cobordant, and that they have the same euler characteristics. 
\end{proof}

Hence, we obtain the following:

\begin{corollary} \label{relatedbylog}
If $X$ and $X'$ are two closed oriented simply-connected homeomorphic $4$-manifolds, then $X'$ can be obtained from $X$ by a sequence of log transforms along tori. 
\end{corollary}

\begin{remark} \label{caution}
Since these results provide answers to the problems 15 and 12 asked by Ron Stern in \cite{S}, we would like to pause here to closely examine how useful such cobordisms are $(1)$ to obtain a possible classification scheme for closed oriented simply-connected smooth $4$-manifolds; and $(2)$ to produce new smooth structures. 

\noindent For $(1)$: Assume for the moment that in our cobordisms the round $2$-handles were attached independently, that is, $\partial_- R_k$ could be isotoped away from $\cup_{\, i < k} \, \partial_+ R_k$ for all $k = 1, \ldots N$. Then all the tori we performed log transforms along would be disjointly embedded in $X$. It would then follow that every closed smooth oriented $4$-manifold could be produced from a standard $4$-manifold $X_0$ which is a connect sum of some copies of $\CP$, $\CPb$, $S^2 \x S^2$ and $K3$ by performing a log transform along a link of tori in $X_0$ (assuming the $11/8$ Conjecture). However, our proof of Theorem \ref{RoundCobordismThm} does not guarantee this at all. On the contrary, the round handles in $W_3$ are attached along tori that appear only after the attachment of the round handles in $W_1$. 

\noindent For $(2)$: In the proof of Theorem \ref{RoundCobordismThm}, the piece $\bar{W}_1$ we got was composed of \emph{trivial} round $2$-handles attached to $X$ in a small ball containing $D^3 \x S^1$ following from the proof of Lemma \ref{KillRound1}. Recall that such a round $2$-handle consists of a $1$-handle and a $2$-handle that could be attached independently. Now, attaching a $2$-handle yields connect summing with $\ste$ or $\sto$ and attaching a $1$-handle yields connect summing sith $S^1 \x S^3$. Thus the effect of attaching a \textit{trivial} round $2$-handle is the same as connect summing with $\ste \# S^1 \x S^3$ or with $\sto \# S^1 \x S^3$. Existence of such summands in the intermediate steps will imply the vanishing of Gauge theoretic invariants, making it impossible to trace the effect of the sugeries on these invariants. (To the authors' knowledge, the only Gauge-theoretic invariant that seems to be sensitive to such connect sums is the one defined in \cite{Sa}.)

\noindent To strike the best possible cobordisms meeting the goals of $(1)$ and $(2)$ above, one can insist that these cobordisms have the properties: $(1')$ $X'$ is obtained from $X$ by performing \textit{simultaneous} log transforms along tori embedded in $X$; and $(2')$ None of the log transforms change the homeomorphism type. We shall note however, all recent constructions of exotic smooth structures where reverse engineering is employed involve log transforms which indeed change the homeomorphism type. Further, we note that by Theorem \ref{towertheorem}, these assumptions would imply that between any two homeomorphic simply-connected $4$-manifolds there is a cobordism that can be given by only one pair of $2$- and $3$-handles. Whether or not this is true is still an open question. 
\end{remark}

In fact, our observation above suggest the following, which is comparable to Asimov's ``Fundamental Lemma of Round Handles'':

\begin{proposition} \label{TrivialLog}
Stabilizing or destabilizing a $4$-manifold with $\ste \# S^1 \x S^3$ or with $\sto \# S^1 \x S^3$ is a log transform. 
\end{proposition}

\begin{proof}
Since stabilizations (resp. destabilizations) correspond to connect summing (resp. removing the connected sum summand) $\ste \# S^1 \x S^3$ or $\sto \# S^1 \x S^3$, we can look at this operation locally.  The second diagram in the first row of the handlebody diagrams in Figure \ref{Simplelog} shows the effect of a log $0$ transform in a \linebreak $4$-ball with surgery curve $\alpha$ parallel to the $m$-framed $2$-handle. After sliding the `large' $0$-framed $2$-handle twice over the $0$-framed $2$-handle attached to the $1$-handle on the top, it slides off from the rest of the diagram. We then cancel the same \linebreak $1$-handle against the $0$-framed $2$-handle, and obtain the third diagram, which represents $\ste \# S^1 \x S^3 \setminus D^4$ if $m$ is even, and $\sto \# S^1 \x S^3 \setminus D^4$ if $m$ is odd. 

\noindent The second row of the Figure \ref{Simplelog} demonstrates the inverse of the above operation. This time we perform a log $1$-transform with surgery curve $\alpha$ linking once with the $1$-handle on the bottom in $(\ste \# S^1 \x S^3) \, \setminus D^4$ if $m$ is even, and in \linebreak $(\sto \# S^1 \x S^3) \,  \setminus D^4$ if $m$ is odd. The second figure is the result of this transform, and the third diagram is obtained after similar handle slides as above. Now it is easy to see that all the handles cancel, yielding $D^4$. Note that other log transforms (including the obvious log $0$ transform) can be performed to realize this inverse operation as well.

\begin{figure}[ht]
\begin{center}
\includegraphics{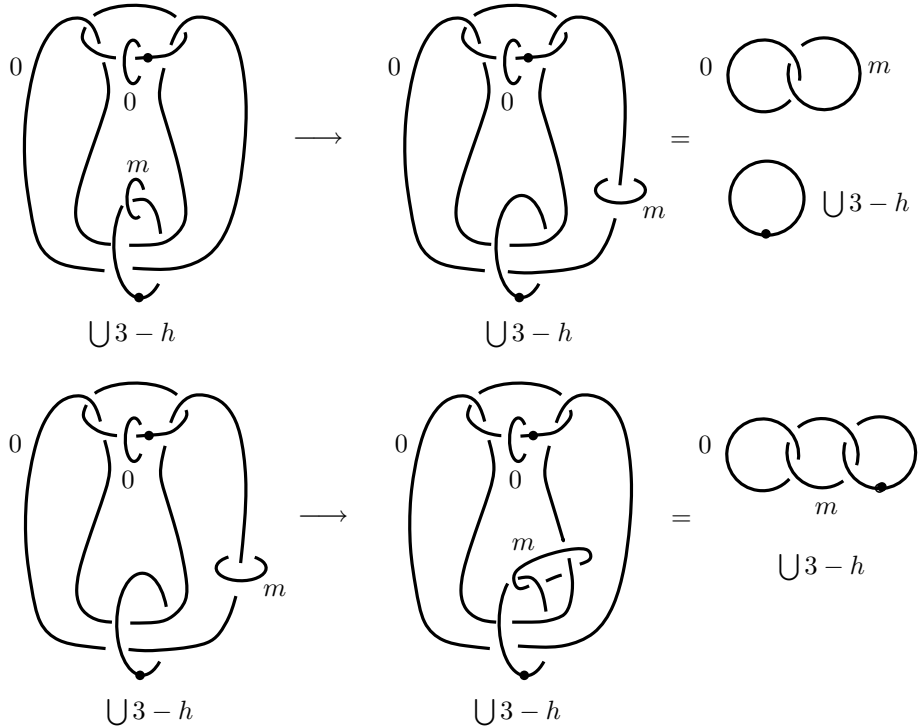}
\caption{\small First row: Log $0$ transform in the $4$-ball resulting in manifolds $(\ste \# S^1 \x S^3) \, \setminus D^4$ for $m: even$ or $(\sto \# S^1 \x S^3) \, \setminus D^4$ for $m: odd$. Second row: A log $1$-transform in $D^4$ as an inverse operation. }
\label{Simplelog}
\end{center}
\end{figure}
\end{proof}

It is worth mentioning that the second diagram in the first row in Figure \ref{Simplelog} union a $4$-handle describes a broken Lefschetz fibration on $\ste \, \# \, S^1 \x S^3$ or on $\sto \, \# \, S^1 \x S^3$ depending on the parity of $m$, where there is only one round singular circle, the higher genus is one, and there exists a section of self-intersection $m$. The second diagram on the second row union a $4$-handle on the other hand describes a similar broken Lefschetz fibration on the $4$-sphere no matter what the parity of $m$ is. (These examples of broken Lefschetz fibrations are due to Auroux, Donaldson and Katzarkov \cite{ADK}, where the description of them via handlebodies we use here were given in \cite{B2}.) In short, these operations are equivalent to performing standard log transforms along regular torus fibers of such broken Lefschetz fibrations.

What follows from Proposition \ref{TrivialLog} is an alternative proof of \, Corollary \ref{relatedbylog} by invoking C.T.C. Wall's celebrated theorem: $X \# \, k \, \ste = X'\# \, k \, \ste$ for large enough $k$, and therefore $X \# \, k \, (\ste \, \# \, S^1 \x S^3)=X' \# \, k \, ( \ste \, \# \, S^1 \x S^3)$. Proposition \ref{TrivialLog} shows that $X$ (resp. X') and $X \# \, k \, (\ste \, \# \, S^1 \x S^3)$ (resp. \linebreak $X'  \# \, k \, ( \ste \, \# S^1 \x S^3)$ are related by $k$ log transforms. Hence one can pass from $X$ to $X'$ by a sequence of $2k$ log transforms. 

Nevertheless, given the lack of any classification scheme for closed oriented simply-connected smooth $4$-manifolds, it is worth noting what we get out of this simple observation, which is a new proof of a theorem of Iwase's \cite{Iwase} (which used a somewhat different language):   

\begin{corollary}[Iwase] \label{simultaneous}
Every closed oriented simply-connected $4$-manifold $X$ can be produced by surgery along a link of tori of self-intersection zero contained in a connected sum of smooth copies of $\CP$, $\CPb$, and $S^1 \x S^3$. Specifically, if $e(X)=a+b+2$, $\sign(X)=a-b$, then there is a link of self-intersection zero tori $L= \sqcup \, T_i$ in $(a+k) \, \CP \# (b+k) \, \CPb \# k \, S^1 \x S^3$ for a large enough integer $k$, such that surgery along this link gives $X$. 
\end{corollary}

\begin{proof}
Let $X_0$ be a connected sum of copies of $\CP$, $\CPb$ with $\eu(X_0)= \eu(X)$ and $\sign(X_0)=\sign(X)$. The latter equality provides us with a cobordism from $X_0$ to $X$. Using the former equality and following the proof of Theorem \ref{RoundCobordismThm} (while skipping the construction of $W_3$), we can produce a cobordism $\tilde{W} = \bar{W}_1 + W'_2$ between $X_0$ and $X$ with round $2$-handles only. Recall that the round $2$-handles $\bar{R}_2^i$ of $\bar{W}_1$ for $i=1, \ldots, k$ are attached independently of each other and the attachment of each one amounts to connect summing the simply-connected $4$-manifold $X_0$ with $\ste \, \# \, S^1 \x S^3$. Then the middle level $\tilde{X}$ of this cobordism at the interface of $\bar{W}_1$ and $W'_2$ is a connected sum of copies of $\CP$, $\CPb$ and $S^1 \x S^3$. Also recall that now the round $2$-handles $R_2^i$ for $i=1, \ldots, k$ are attached to $\tilde{X} = \partial_+(\bar{W}_1)$ independently as well, and thus, we can realize all these round $2$-handle attachments as a simultaneous log transforms along embedded tori in $\tilde{X}$. It follows that $X$ can be obtained from a connected sum of $a \, \CP \# b \, \CPb$ with $k$ copies of $\sto \# S^1 \x S^3$, where $e(X)=e(X_0)=a+b+2$, $\sign(X)=\sign(X_0)=a-b$.
\end{proof}

\noindent If the number $k$ could be determined by the intersection form of $X$ alone, this would provide us with a standard manifold $X_Q$ from which one could obtain $X$ with $Q_X = Q$ by a surgery along a framed link of tori in $X_Q$. So it is natural to ask:

\begin{question}
Let $Q$ be a fixed intersection form of a closed oriented simply-connected smooth $4$-manifold. Let $min(X,Y)$ be the minimum number of $2$-handles needed in an h-cobordism between two closed oriented simply-connected smooth \linebreak $4$-manifolds $X$ and $Y$. Is there an upper bound on $\{min(X,Y) | Q_X=Q_Y=Q \}$? 
\end{question}   

We continue with some other applications of our results above. 

\emph{Every construction} of an infinite family of mutually non-diffeomorphic closed smooth oriented simply-connected $4$-manifolds in the same homeomorphism class \emph{given to date} involve log transforms. Chronologically, the first constructions of such families were obtained by standard (thus integral) logarithmic transforms along homologically essential tori in elliptic surfaces, which were then followed by applications of the knot surgery operation of Fintushel-Stern, and finally by log transforms along null-homologous tori in an exotic copy of a standard $4$-manifold. (See for instance \cite{FS6}.) In all of these cases one obtains an infinite family of non-diffeomorphic \linebreak $4$-manifolds for which the following open problem can be tested: Do all homeomorphic simply-connected $4$-manifolds become diffeomorphic after connect summing each with $S^2 \x S^2$ (or $\sto$) once? 

If \, $\{X_m \, | \, m: 1, 2, \ldots \}$ is a family of mutually non-diffeomorphic closed smooth oriented simply-connected $4$-manifolds in the same homeomorphism class constructed using the first or the third approach discussed in the previous paragraph, we see that each $X_{m+1}$ is obtained from $X_m$ by a log $\pm 1$ transform. In this case one can use the Morgan-Mrowka-Szabo gluing formula to compare their Seiberg-Witten invariants. On the other hand, the \textit{knot surgery operation} was defined as follows \cite{FS1}: Let $T$ be a torus with a trivial tubular neighborhood $N(T)$ in a simply-connected $4$-manifold $X$ with simply-connected complement and $K$ be a knot in $S^3$. Then define $X_K= X \setminus N(T) \cup_{\phi} S^1 \times (S^3 \setminus N(K))$, where $N(K)$ is a tubular neighborhood of $K$ in $S^3$ and the boundary diffeomorphism $\phi$ is chosen so that the resulting manifold $X_K$ is also simply-connected. The authors, by giving an elegant formula for the change in Seiberg-Witten invariants in terms of the Alexander polynomial of the knot $K$, showed that infinitely many exotic smooth structures can be produced on any $4$-manifolds $X$ with non-trivial Seiberg-Witten invariants and with such embedded tori $T$. The important observation built into the proof is that from any knot surgered $4$-manifold $X_K$ one can obtain $X$ back by log transforms along null-homologous tori: Say that $K \subset S^3$ can be unknotted by changing $n$ crossings. Then there is a sequence of knots $K = K_0, K_1, \ldots ,K_n$ with $K_n$ the unknot, and a collection of loops  $\bigsqcup_{i=1}^n \alpha_i \subset S^3 \backslash N(K)$ such that by blowing up along these loops, we progressively unknot $K$: i.e. if we blow up along $\bigsqcup_{i=1}^m\alpha_i$ then we get $S^3 \backslash K_m$. We can consider the $S^1 \times \alpha_i$ as tori in $S^1 \times S^3 \backslash N(K)$ and hence as tori in  $X \backslash N(T) \cup (S^1 \times S^3\backslash N(K))$. Performing $(\pm 1)$-log transforms on the first $i$ tori in $X_K$, we get the manifold $X_{K_i}$, thus we get $X$ after performing all the log transforms. Each log transform gives us knot surgery on $X$, with one more crossing undone, that is. Furthermore, the assumptions made on the complement of $T$ in $X$ guarantee that in each step we get a simply-connected $4$-manifold. Hence, by Theorem \ref{towertheorem} we obtain:

\begin{corollary} \label{logrelevance}
\emph{Every} infinite family of mutually non-diffeomorphic closed smooth oriented simply-connected non-spin $4$-manifolds in the same homeomorphism class \emph{constructed using a sequence of logarithmic transforms as above} consists of members that become diffeomorphic after one stabilization with $\ste$ or with $\sto$. The same holds in the spin case, if one stabilizes with $\ste \# \CPb$.
\end{corollary}

\noindent As we summarized above, \emph{all} constructions of such infinite families of mutually non-diffeomorphic closed smooth oriented simply-connected manifolds we know of satisfy the statement of this corollary.
 
\begin{remark} Using heroic Kirby calculus, Auckly in \cite{Au} (for $\sto$) and Akbulut in \cite{Ak} (for $\ste$) proved the above theorem for knot surgered $4$-manifolds. Here we have obtained a new proof of their result (which is weaker in the spin case). The assumptions Auckly and Akbulut make in their papers on the boundary diffeomorphisms $\phi$ and the existence of bounding disks in the complement of the torus $T$ with certain framings not only guarantee that the result of the sugery is again simply-connected, but also allow the authors to restrict their attention to connect sums with only $\ste$ or with only $\sto$, respectively. Our imprecise choice of $\ste$ or $\sto$ in the above statement is due to us not making any extra assumptions on such disks (or more precisely, on their relative framings), which in particular is the reason why we get a weaker result in the case of spin $4$-manifolds. 
\end{remark}

Following Fintushel-Stern's conjecture that $X_K = X_{K'}$ if and only if $K=\pm K'$ (the negative sign means the `mirror' of $K$), one might ask whether the unknotting number of $K$ gives a lower bound on the number of log transforms required to pass from $X_K$ to $X_K'$. We see that it does not:

\begin{corollary} 
Let $X$ be a closed oriented simply-connected $4$-manifold, $T$ be an embedded torus in $X$ with trivial normal bundle, and $X_K$ be the manifold obtained by knot surgery in $X$ along $T$. Then, for any two knots $K$ and $K'$, $X_K$ and $X_K'$ are related via two log transforms along tori. 
\end{corollary}

\begin{proof}
As seen in the proof of Corollary \ref{logrelevance}, there is a cobordism from $X_K$ to $X_{K'}$ with at most one pair of $2$- and $3$-handles when the manifolds are non-spin. We can modify this cobordism as in the proof of Theorem \ref{RoundCobordismThm} by introducing at most one canceling pair of $1$- and $2$-handles, and hence getting a round cobordism with at most two round $2$-handles between $X_K$ and $X_{K'}$. In the spin case, we can realize the cobordism with at most two log transforms using stabilization with $S^2 \x S^2 \# \, S^1 \x S^3$ as before.  
\end{proof}

\begin{remark}
The proof of this corollary mutadis mutandis gives that if the answer to Stern's Problem 14 were `yes', then any two $X$ and $X'$ would be related by at most two log transforms. 
\end{remark}

\vspace{0.1in}
\section{Broken Lefschetz fibrations and logarithmic transforms} \label{BLFs}

Since every closed smooth oriented $4$-manifold admits a broken Lefschetz fibration, it is natural to ask whether exotic smooth structures on a fixed homeomorphism class of a $4$-manifold can be related through modifications of broken Lefschetz fibrations. One might have various different cobordisms between two homeomorphic $4$-manifolds, and thus, there are various possible ways to realize such modifications. Underlying our preference for the ``fibered'' operations presented below are the types of cobordisms discussed in the previous section.


We first show that a generalized logarithmic transform is a standard logarithmic transform along a torus fiber component of a generalized fibration. Namely:

\begin{lemma} \label{LogAlongBLF}
Let $X$ be a closed smooth oriented $4$-manifold and $T_1, \ldots, T_n$ be disjointly embedded self-intersection zero tori in it. For any prescribed surgery composed of generalized logarithmic $p_i$-transform along $T_i$, for $i=1, \ldots, n$, there exists a broken Lefschetz fibration $f: X \to S^2$ with respect to which the surgery can be realized as a standard logarithmic $p_i$-transform along elliptic fiber components $T_i$ for all $i= 1, \ldots, n$ simultaneously.
\end{lemma}

\begin{proof}
Let $\nu T_i$ be the normal bundle of the torus $T_i$ with framing $\pi_i:\nu T_i \to D^2$ for $i=1, \ldots, n$. Setting $N= \sqcup_{i=1}^n \nu T_i$, we have the map $\pi := \sqcup_{i=1}^n \pi_i: N \to D^2$. 

Let $r: D^2 \to S^2$ be the quotient map defined by collapsing $\partial D^2$ to a point. Assume that $r( \partial D^2) = \{\text{NP}\}$ and $r(0) = \{\text{SP}\}$. The composition $r \circ \pi$ is a surjective map from $N$ to $S^2$, which we can extend to all of $X$ by mapping $X \setminus N$ to $\{\text{NP}\}$ so as to get a surjective continuous map $g: X \to S^2$, which is smooth away from $g^{-1}(\{\text{NP}\})$. Letting $N_0$ be the preimage of a smaller disk neighborhood of the southern hemisphere under $g$, we can approximate  $g$ by a generic map $h: X \to S^2$ relative to $N_0$, which can then be modified to have only indefinite singularities using Saeki's construction \cite{Sae}. In turn we obtain a broken Lefschetz fibration \linebreak $f: X \to S^2$, where the \textit{framed tubular neighborhood} $N_0$ is the tubular neighborhood of the fiber $f^{-1}(\text{\{SP\}})$ as shown in \cite{B}, containing all $T_i$ as fiber components. Note that these fiber components can be null-homologous. Hence the prescribed surgery amounts to performing standard logarithmic transforms along the fiber components $T_i$ of $f: X \to S^2$.
\end{proof}

We are now ready to discuss some instances.

\begin{theorem} \label{BLFcobordismThm}
Let $X$ and $X'$ be two closed smooth oriented $4$-manifolds with the same euler characteristic and signature. \\
\indent \emph{(i)} If $X$ and $X'$ are simply-connected and equipped with broken Lefschetz fibrations $f: X \to S^2$ and $f': X' \to S^2$ respectively, then the latter can be obtained from the former via a sequence of modifications of broken Lefschetz fibrations, corresponding to logarithmic transforms and homotopies of broken Lefschetz fibrations. \\
\indent \emph{(ii)} If $X'$ is obtained from $X$ by performing generalized logarithmic $p_i$ transforms along disjointly embedded tori $T_i$ of self-intersection zero in $X$, for $i=1, \ldots, m$, then there exists a broken Lefschetz fibration $f': X' \to S^2$ obtained from a broken Lefschetz fibration $f: X \to S^2$ by standard logarithmic $p_i$-transforms along elliptic fiber components. 
\end{theorem}

\begin{proof}
Part (i): The results in the previous section show that there exists a trivial round cobordism $W \cup_{\phi} W'$ between $X$ and $X'$ such that $W$ is a cobordism from $X$ to $\hat{X}=X \# m \, (S^2 \x S^2 \# S^1 \x S^3)$ and $W'$, upside down, is a cobordism from $X'$ to $\hat{X'}=X' \# m \, (S^2 \x S^2 \# S^1 \x S^3)$, where $\phi: \hat{X} \to \hat{X'}$ is an orientation-reversing diffeomorphism. We can take the connected sum of the broken Lefschetz fibration $f: X \to S^2$ with the standard broken Lefschetz fibration on $S^2 \x S^2 \# S^1 \x S^3$ repeatedly $m$-times to get a broken Lefschetz fibration $\hat{f}: \hat{X} \to S^2$ (see \cite{B2} and Figure \ref{Simplelog} above), and similarly we can get a broken Lefschetz fibration $\hat{f'} : \hat{X'} \to S^2$. The latter, precomposed with $\phi$ gives a broken Lefschetz fibration $\phi \circ \hat{f'}: \hat{X} \to S^2$. By William's theorem from \cite{W}, these two fibrations on $\hat{X}$ are related via a sequence of moves between broken Lefschetz fibrations, concluding the statement of part (i).

Part (ii) follows from Lemma \ref{LogAlongBLF} above. The cobordism from $X$ to $X'$ is given by standard logarithmic $p_i$ transforms along $T_i$, for $i=1, \ldots, n$, yielding a new broken Lefschetz fibration $f': X' \to S^2$ with multiple fiber components $T_i$. Around each mutiple torus fiber we can replace the fibration $D^2 \x T^2 \to D^2$ with a multiple torus fiber over $0 \in D^2$ with a broken Lefschetz fibration. The existence of such a broken Lefschetz fibration is provided by Gay-Kirby's result where the authors show how to extend any circle valued morse function without extrema defined on the boundary of an arbitrary compact oriented $4$-manifold $Z$ to a broken Lefschetz fibration over $D^2$ on $Z$. \cite{GK}
\end{proof}

\begin{remark} \label{OtherBLFcobordisms}
Regarding part (i) of the theorem: A simpler cobordism between $X$ and $X'$ could be given by a sequence of $m$ $2-$ and $m$ $3-$ handle attachments, which correspond to connect summing $X$ (resp. $X'$) with $m$ copies of $S^2 \x S^2$ to get the middle manifold $\tilde{X}$ (resp. $\tilde{X}'$). If we start with two broken Lefschetz fibrations $f: X \to S^2$ and $f': X \to S^2$, then connect summing these fibrations with the standard fibration on $S^2 \x S^2$ $m$-times, and precomposing $f'$ with the diffeomorphism $\phi: \tilde{X} \to \tilde{X}'$ given by the cobordism, we get two broken Lefschetz fibrations $\tilde{f}$ and $\tilde{\phi} \circ \tilde{f}': \tilde{X} \to S^2$ as in the proof of Theorem \ref{BLFcobordismThm}. Now, William's result can be applied to relate these two broken Lefschetz fibrations by a sequence of fibered modifications, and thus giving yet another way to pass from $(X, f)$ to $(X', f')$ through modifications of broken Lefschetz fibrations.  

\noindent Regarding (ii): The assumptions of this theorem are the same as those given in the proposed extra assumption (1') on logarithmic transforms discussed in the previous section. 
\end{remark}

\begin{remark}
Let $W = I \x X$ be a \textit{trivial} cobordism, where $X$ is a closed, not necessarily simply-connected $4$-manifold. Gay and Kirby have recently constructed indefinite generic maps over $I \x S^2$ on $W$, connecting two prescribed broken Lefschetz fibrations (perturbed to indefinite generic maps) over $S^2$ on $\{0\} \x X$ and on $\{1\} \x X$. It would be interesting to see whether their arguments can be adapted to non-trivial round cobordisms we have considered in this article so as to improve our Theorem \ref{RoundCobordismThm}. 
\end{remark}


\vspace{0.2in}
\noindent \textit{Acknowledgements.} We would like to thank Dave Auckly for clarifications on his earlier work, Danny Ruberman and Ron Stern for their comments on a draft of this paper. We are greateful to Motoo Tange for bringing Iwase's result \cite{Iwase} to our attention, which we were not aware of at the time of posting the first version of our article. The first author was partially supported by the NSF grant DMS-0906912.


\end{document}